\documentclass[12pt,a4paper]{amsart}
\usepackage{amssymb,amsmath,amsfonts}
\usepackage[usenames,dvipsnames,svgnames,table]{xcolor}
\usepackage{amsmath,amsfonts,amsthm}
\usepackage{amssymb,latexsym}
\usepackage{setspace}
\usepackage{ams thm,ams symb,amsmath}
\usepackage[all]{xy}
 \usepackage{graphicx}
 \usepackage{enumerate}
\theoremstyle{plain}
\usepackage[square,numbers]{natbib}
\headsep=1.2500cm \numberwithin{equation}{section}
\newtheorem{Theorem}{Theorem}[section]
\newtheorem{Proposition}[Theorem]{Proposition}
\newtheorem{cor}[Theorem]{Corollary}
\newtheorem{Lemma}[Theorem]{Lemma}
\theoremstyle{remark}
\newtheorem{Definition}[Theorem]{Definition}
\newtheorem{Example}[Theorem]{Example}

\newtheorem{Remark}[Theorem]{Remark}

\begin{document}

\title[essential contractibility of  
 Banach algebras]{essential   contractibility of \\ 
 Banach algebras}

 \author[B. S. Mortazavi-Samarin]{Batoul S. Mortazavi-Samarin}
 \address{Department of Mathematics and Computer Science, Amirkabir University of Technology (Tehran
Polytechnic), Iran.}
  \email{bmortazavi@aut.ac.ir}

 \author[M. Rostami]{Mehdi Rostami}
 \address{Department of Mathematics and Computer Science, Amirkabir University of Technology (Tehran
Polytechnic), Iran.}
  \email{mross@aut.ac.ir}

\author[A. Sahami]{Amir Sahami}
\address{Faculty of Basic Sciences, Department of Mathematics, Ilam University, P.O.Box 69315-516, Ilam,
	Iran.}
	\email{a.sahami@ilam.ac.ir}
	
\keywords{essential contractibility,extension, ideal, symmetric abstract Segal algebra}

\subjclass[2010]{Primary 43A20, 46H25, Secondary 46H05.}

\maketitle
%%%%%%%%%%%%%%%%%%%%%%%%%%%%%%%%%%%%%%%%%%%%%%%%%%%%%%%%%%%%%%%%%%%%%%%%%%%%%%%%%%%%%%%%%%%%%%%%%%%%%%%%%%%%%%%%%%%%%%%%%%%%%%%%%%%%%%%%%%%%%%%%%%%%%%%%%%
\begin{abstract}
In this paper, a new notion, essential contractibility of Banach algebras, is introduced and some of its properties are examined.	The main result is to investigate the essential contractibility of (symmetric abstract) Segal algebras. 
\end{abstract}
%%%%%%%%%%%%%%%%%%%%%%%%%%%%%%%%%%%%%%%%%%%%%%%%%%%%%%%%%%%%%%%%%%%%%%%%%%%%%%%%%%%%%%%%%%%%%%%%%%%%%%%%%%%%%%%%%%%%%%%%%%%%%%%%%%%%%%%%%%%%%%%%%%%%%%%%%%
\section{Introduction}
One of the basically cohomological Banach algebras is contractible Banach algebras. 
A Banach algebra $A$ is said to be contractible when every bounded derivation $D:A\rightarrow X$ is inner for every Banach $A$-bimodule $X$. Alternatively, $A$ is contractible if it possesses a diagonal, i.e., an element $u\in A\widehat{\otimes}A$ satisfying $au=ua$ and $\pi_A(u)a=a$ for all $a\in A$. Curtis and Loy showed in \cite{cl} that any commutative contractible Banach algebra must be finite dimensional. However, finding an example of an infinite-dimensional contractible Banach algebra remains an open challenge.
Another key homological notion in Banach algebras is amenability which was first introduced
and studied by Johnson \cite{Johnson}. A stronger notion that amenability, known as essential amenability, was later proposed by Ghahramani and Loy in \cite{Ghahramani.Loy}.
They defined a Banach algebra $A$ as essentially
amenable if $\mathcal{H}^1(A,{X}^\ast) = \{0\}$, where ${X}$ is  a
neo-unital Banach $A$-bimodule, i.e., every continuous derivation
from A into the dual of every Banach $A$-bimodule ${X} = A\cdot {X}
\cdot A,$ is inner. A criterion for
the essential amenability of symmetric Segal algebras under certain
conditions was established by Ghahramani and Loy \cite{Ghahramani.Loy}. Building on this, Samea extended these ideas to the broader setting involving abstract Segal algebras in
\cite{Samea}.

The paper is organized as follows. Section \ref{sec2} is devoted to
introducing the concept of essential contractibility of Banach algebra, along with relevant examples. Furthermore, infinite-dimensional examples of essentially contractible Banach algebras are presented to highlight the significance and potential of this concept for further investigation. In section \ref{sec3}, some hereditary properties of essential contractibility are examined.
In section \ref{sec4}, the essential contractibility of Segal Banach algebra is investigated.
In addition, several criteria for the essential contractibility of
symmetric Segal  Banach algebras are provided.
%%%%%%%%%%%%%%%%%%%%%%%%%%%%%%%%%%%
%%%%%%%%%%%%%%%%%%%%%%%%%%%%%%%%%%%
%%%%%%%%%%%%%%%%%%%%%%%%%%%%%%%%%%%
\section{Essential contractibility}\label{sec2}
Let $A$ be a Banach algebra. For $n\in \mathbb{N},$ define $${A^n
	=\left\{{\prod}_{i=1}^n a_1\dots a_n : a_i\in A, 1\leq i \leq n\right\}}.$$ We denote by ${[A^n]}$ the linear span of $A^n$. 
A Banach  $A$-bimodule ${X}$ is called essential if $X =\overline{[A\cdot X\cdot A]}$
where $A\cdot X\cdot A=\{a\cdot x \cdot b :a,b\in A,x\in {X} \}$. Additionally, $X$
is said to be neo-unital if $X=A\cdot X\cdot A$. According to the Cohen factorization theorem \cite[Theorem
32.22]{Hewitt.Ross}, every essential  Banach $A$-bimodule $X$ is
neo-unital whenever $A$ has a bounded approximate identity.

Now, Let $A$ be a Banach algebra and let ${X}$ be a Banach $A$-bimodule.  A continuous linear map $D:A\rightarrow {X}$ is called a derivation if it satisfies
$$D(ab)=a\cdot D(b)+D(a)\cdot b,$$
for all $a, b\in A$.
For a given $x\in X$, define a continuous linear map ${ad}_x:A\rightarrow X$ by
$${{ad}}_{x}(a)=a\cdot x-x\cdot a, \quad (a\in A).$$
This map, known as an inner derivation, satisfies the derivation property.
\begin{Definition}
	A Banach algebra $A$ is essentially contractible if
	for every neo-unital Banach $A$-bimodule $X$, every derivation $D:A\rightarrow X$ is inner.
\end{Definition}
Let $A$ be a Banach algebra and let $\Phi(A)$ represent the character space of
$A$, consisting of all non-zero multiplicative linear functionals on $A$. For $\varphi\in\Phi(A)$, a Banach $A$-bimodule
$X$ is called left $\varphi$-linked if the left module action is
induced by $\varphi$; i.e., $$a \cdot x = \varphi(a)x,$$ for all $a\in A$ and $x\in X$.
Such an $A$-bimodule is denoted by $_{\varphi}X.$ In the same manner,
a right $\varphi$-linked Banach $A$-bimodule is defined and it is denoted by $X_{\varphi}.$
We also use $_{\varphi}X_{\varphi}$ to indicate an $A$-bimodule that is both left and right $\varphi$-linked.
\begin{Remark}\label{7}
	Let $A$ and $\mathcal X$ be a Banach algebras and let $\varphi \in \Phi(A)$ be a fixed
	character. The Banach algebra
	$_{\varphi}{\mathcal{X}}_{\varphi}$ becomes a neo-unital
	$A$-bimodule. Specifically, for some $b\in A$ with $\varphi(b)\neq0$, we have $x = \dfrac{b}{\varphi(b)} \cdot x \cdot
	\dfrac{b}{\varphi(b)}$, for every $x\in{_{\varphi}{\mathcal{X}}_{\varphi}}$. Any continuous derivation from $A$ into
	$_\varphi\mathbb{C}_\varphi$ is referred to as a point derivation. This implies that every essentially contractible Banach algebra has no non-zero
	point derivation. This criterion provides us with numerous examples of Banach algebras that are not essentially contractible.
	As an example, consider the disc algebra $A(\overline{\Bbb D})$, defined as 
	$$A(\overline{\Bbb D})=\{f\in C(\overline{\Bbb D}): f|_{\Bbb D}~{\rm is~analytic}\},$$
	where $\Bbb D$ denotes the open unit disc in $\Bbb C$.
	For a point $z_0$ on the boundary of $\bar{\Bbb D}$,
	and the character $\varphi(f)=f(z_0)$, the linear functional $D:A(\overline{\Bbb D})\rightarrow A(\overline{\Bbb D})$ given by $D(f)=f^{\prime}(z_0)1$
	is a point derivation on $A(\overline{\Bbb D})$, where $f^{\prime}(z_0)$ is understood in the sense of the derivative of $f$ at $z_0$.
	\end{Remark}
%%%%%%%%%%%%%%%%%%%%%%%%%%%%%%%%%%%%%%%%%%%%%%%%%%%%%%%%%%%%%%%%%%%%%%%%%%%%%
For a Banach algebra $A$,
the projective tensor product of $A$ with itself is denoted by $A\widehat{\otimes}A$ and the diagonal operator from $A\widehat{\otimes}A$
into $A$ is defined by $\pi_A(a\otimes b)=ab$.
It is obvious that every contractible Banach algebra is essentially contractible. In the following, it shown that the essential
contractibility with the condition of unitality is the same as the contractibility.
\begin{Theorem}\label{31}
	Every essentially contractible Banach algebra with a bounded central approximate identity is pseudo-contractible.
\end{Theorem}
\begin{proof}
	Let $(e_{\alpha})$ be the bounded central approximate identity of $A$. Consider the kernel of the diagonal map $\pi_A$, denoted by $\ker\pi_A\subseteq A\widehat{\otimes}A$, and equip it with the usual Banach $A$-bimodule structure.
	Since $a\otimes b=\lim_{\alpha}e_{\alpha}(a\otimes b)e_{\alpha}$, for every $a\otimes b\in\ker\pi_A$, it follows from the Cohen factorization theorem that $\ker\pi_A$ is
	a neo-unital Banach $A$-bimodule.
	We set $D_{\alpha}={ad}_{ M_{\alpha}}:A\rightarrow A\widehat{\otimes}A$, where $M_{\alpha}:=e_{\alpha}\otimes e_{\alpha}.$ Thus $\text{Im}(D)\subseteq\ker\pi_A$.
	This implies the existence of $N_{\alpha}\in\ker\pi_A$ such that ${ad}_{M_{\alpha}}={ad}_{ N_{\alpha}}$.
	Therefore, $(M_{\alpha}- N_{\alpha})$ forms a central approximate diagonal for $A$, as it satisfies $$a \cdot (M_{\alpha}-N_{\alpha})=(M_{\alpha}-N_{\alpha})\cdot a \qquad{\rm and}\qquad a\cdot \pi_A(M_{\alpha}-N_{\alpha})\rightarrow a$$ for all $a\in A$.
\end{proof}
%\begin{Remark}
	%Let $G$ be a nondiscrete amenable locally compact group. If $A(G)$ is essentially contractible,
	%then by using Theorem \ref{31}, we can conclude that 
	%$A(G)$ is pseudo-contractible. However, this leads to a contradiction, as shown in \cite[Theorem 4.2]{sss}. Therefore, $A(G)$ cannot be essentially contractible.
	%\end{Remark}
Applying Theorem \ref{31} and \cite[Theorem 2.4]{gz} we deduce the
following result.
\begin{cor}\label{3}
	Every essentially contractible Banach algebra with unit is contractible.
\end{cor}
A direct consequence of the above corollary is that the discrete
group algebra $\ell^1(G)$ is essentially contractible if and only if
$G$ is finite.
\begin{Remark}
	Consider the semigroup $S=(\Bbb N,\max)$. Then, if $\ell^1(S)$ is essentially
	contractible, then by using Corollary \ref{3}, $\ell^1(S)$ is
	contractible. This implies that $\ell^1(S)$ is amenable and consequently the semigroup $S$
	has finitely many idempotents \cite[Theorem 2]{dp}, that is a contradiction. Therefore,
	$\ell^1(S)$ is not essentially contractible.
\end{Remark}
%%%%%%%%%%%%%%%%%%%%%%%%%%%%%%%%%%%%%%%%%%%%%%%%%%%%%%%%%%%%%%%%%%%%%%%%%%%%%
In the sequel, some examples that distinguish the essential contractibility from the other concepts of cohomological notions are presented.
\subsection*{Examples}
\noindent 
\begin{Example}
	An example of  (infinite dimensional) Banach algebras that is essentially contractible but not contractible.
\end{Example}
Let $A$ be a nonzero Banach space. By defining the product $a\cdot b=0$ for all
$a,b\in A,$ every neo-unital Banach A-bimodule becomes a singular zero
set. In fact, $X = A\cdot X\cdot A=A\cdot(A\cdot X\cdot A)\cdot A
=\{0\}$. This shows that $A$ is essentially contractible. However, $A$ is not contractible, as it does not have a unit element.
\begin{Example}\label{21}
	An example of a Banach algebra that is essentially contractible but not amenable (contractible).
\end{Example}
Let $S$ be a nonempty set with $|S|\geq2$. 
\begin{enumerate}[a)]
\item
For fixed an element $s_0\in S$, define a semigroup structure on $S$ using the multiplication
$st=s_0$ for all $s, t\in S$. We show that $\ell^1(S)$  is essentially
contractible. Let $X$ be a neo-unital Banach $\ell^1(S)$-bimodule and $D
:\ell^1(S)\rightarrow X$ be a continuous derivation.
It is straightforward to see that $f\ast g=(\sum_{st=s_0}f(s)g(t))\delta_{s_0}=\Bbb C\delta_{s_0}$, for all $f, g\in\ell^1(S)$.  Since $X$ is neo-unital we get
$$X=\ell^1(S)X\ell^1(S)=\ell^1(S)(\ell^1(S)X\ell^1(S))\ell^1(S)=\Bbb C\delta_{s_0}X\Bbb C\delta_{s_0}=\delta_{s_0}X\delta_{s_0}.$$
Let $f\in\ell^1(S)$ and $x\in X$. Then for some $y\in X$ we have $x=\delta_{s_0}y\delta_{s_0}$, leading to
\begin{equation*}
		\begin{split}
			f\cdot x&=f.(\delta_{s_0}y\delta_{s_0})=(f\ast\delta_{s_0})y\delta_{s_0}\\&=(\sum_{s\in S}f(s))x=\delta_{s_0}y(\delta_{s_0}\ast f)\\&=(\delta_{s_0}y\delta_{s_0})\cdot f=x\cdot f.
		\end{split}
	\end{equation*}

Specifically, $\delta_{s_0}\cdot x=x\cdot\delta_{s_0}=x$. Additionally, we find $$D(\delta_{s_0})=D(\delta_{s_0}\ast\delta_{s_0})=\delta_{s_0}D(\delta_{s_0})+D(\delta_{s_0})\delta_{s_0}=2D(\delta_{s_0}),$$
which implies $D(\delta_{s_0})=0$. Consequently,
$D(f\ast\delta_{s_0})=D(f)\delta_{s_0}+fD(\delta_{s_0})=D(f)$. On the other hand, $D(f\ast\delta_{s_0})=(\sum_{s\in S}f(s))D(\delta_{s_0})=0$ and thus $D=0$.
This shows that $\ell^1(S)$ is essentially contractible. However, $\ell^1(S)$ is not amenable, as it lacks a bounded approximate identity.
\item
An element $z$ in  $S$ is called  left (right) zero if $zs = z~ (sz = z)$ for all $s \in S.$ If every member of S is a left (right) zero, then S is called a left (right) zero semigroup. A left zero semigroup that is also a right zero semigroup is called a zero semigroup. If S has a zero 0 and st = 0 for all $s, t \in S,$ then $S$ is called a null semigroup. Suppose that $S$ is a topological null semigroup. By \cite[Theorem 3.4]{BS1} $M(S)$ does not have a left or a right approximate identity. We show that $M(S)$   is essential contractible. Let $ X $ be a neo-unital $ M(S)$-bimodule and $D :M(S)\rightarrow X$ be a continuous  derivation. So $D = 0.$ In fact, by \cite[Lemma 3.3]{BS1} we have
$$\mu(S)D(\delta_0) = D(\mu(S)\delta_0) = D(\mu \ast  \delta_0)	= D(\mu).\delta_0 + \mu.D(\delta_0)
= D(\mu) + \mu(S)D(\delta_0).$$
for all $\mu \in M(S)$.
\end{enumerate}
\begin{Example}
An essentially amenable which is not essentially contractible.  
\end{Example}
Let $G$ be a non-compact amenable  locally compact group. The group algebra $L^1(G)$ is essentially  amenable \cite[Corollary 7.1]{Ghahramani.Loy}. These group algebras can not be essentially  contractible. In fact, if $L^1(G)$ is essentially  contractible, then it is  already left $ \varphi$-essentially  contractible for every $\varphi \in \Phi(L^1(G))\cup \{0\}.$ So $G$ must be compact \cite[Theorem 2]{Sahami}.
$ {} $\\
%%%%%%%%%%%%%%%%%%%%%%%%%%%%%%%%%%%%%%%%%%%%%%%%%%%%%%%%%%%%%%%%%%%%%%%%%%%%%
Further examples will be introduced progressively in the following sections.
%%%%%%%%%%%%%%%%%%%%%%%%%%%%%%%%%%%%%%%%%%%%%%%%%%%%%%%%%%%%%%%%%%%%%%%%%%%%%%%%%%%%%%%%%%%%%%%%%%%%%%%%%%%%%%%%%%%%%%%%%%%%%%%%%%%%%%%%%%%%%%%%%%%%%%%%%%
\section{Properties of essential contractibility}\label{sec3}
We will now establish some basic properties of essential contractibility.
\begin{Proposition}\label{11}
	Let $A$ be an essentially contractible Banach algebra and
	$\Psi:A\rightarrow B$ be a continuous epimorphism. Then $B$ is
	essentially contractible.
\end{Proposition}
\begin{proof}
	Let $X$ be a neo-unital Banach $B$-bimodule and $D:B \rightarrow X$ be a bounded derivation. The following module actions
	$$a \cdot x=\Psi(a)x,\quad x\cdot a=x\Psi(a)\qquad (a\in A, x\in X)$$
	make $X$ into a neo-unital Banach $A$-bimodule. Then the map $D\circ\Psi:A\rightarrow X$ is a bounded derivation and since $D\circ\Psi$ is inner, we conclude that $D$ must be inner.
\end{proof}
%%%%%%%%%%%%%%%%%%%%%%%%%%%%%%%%%%%%%%%%%%%%%%%%%%%%%%%%%%%%%%%%%%%%%%%%%%%%%
Utilizing Theorem \ref{3} and Proposition \ref{11}, we investigate the connection between the finiteness of the underlying group and
essentially contractibility of some associated Banach algebras.

Let $A$ be a Banach algebra. The second dual of $A$, equipped with the first Arens product $"\Box"$, can be considered as a Banach algebra
with the properties $a\Box F=a\cdot F$ and $F\Box a=F\cdot a$ for each $a\in A$ and $F\in A^{\ast\ast}$ (see \cite[Theorem 2.6.15]{dales}).
Let $Y$ be a closed $L^1(G)$-submodule of $L^\infty(G)$. If $F\cdot f \in Y$ for all $F\in L^\infty(G)^{\ast}$ and $f\in Y,$ then $Y$ is called left introverted.
So $Y^\ast$ forms a Banach algebra that its multiplication is induced by the first Arens product. A continuous bounded function $f\in C(G)$ is
called almost periodic if the set $\{\ell_sf: s\in G\}$ is relatively norm compact, where $\ell_sf:G\rightarrow\Bbb C$ denotes the left
translation of $f$ given by $\ell_sf(t)=f(st)$ for all $t\in G$. The space of all almost periodic functions on
$G$ is denoted by $ AP(G),$ which is an example of such $Y.$ We denote $ \dot{Y} $ as a left introverted subspace of $L^\infty(G)$,
where $G$ is a maximally almost periodic (MAP) locally compact group and $AP(G) \subset \dot{Y}$ (For more information on MAP groups see \cite{pal}).
For a locally compact group $G$,
two subbimodules of $L^\infty (G)$, consisting of left uniformly continuous functions and those vanishing at infinity, are denoted by $LUC(G)$ and $L_0^\infty(G),$ respectively.
\begin{Proposition}
	Let $G$ be a locally compact group. The following assertions are equivalent.
	\item[(i)] $M(G)$ is essentially contractible.
	\item[(ii)] ${LUC(G)}^\ast$ is essentially contractible.
	\item[(iii)] $ {L^1(G)}^{\ast\ast}$ is essentially contractible.
	\item[(iv)] ${L_0^\infty(G)}^\ast$ is essentially contractible.
	\item[(v)] ${\dot{Y}}^\ast$ is essentially contractible.
	\item[(vi)] $G$ is finite.
\end{Proposition}
\begin{proof}
	If $G$ is finite then
	$M(G)=LUC(G)^*=L^1(G)^{**}=L_0^{\infty}(G)^*={\dot{Y}}^\ast=\ell^1(G)$
	are contractible and then they are essentially contractible.
	
	\noindent
	$\rm{(i)}$
	If $M(G)$ is essentially contractible, since it is unital, by Theorem \ref{3} $M(G)$ is contractible and by \cite{Spronk}, $G$ is finite.
	\\ $ \rm{(ii)}$ Since the restriction map $\Psi :{LUC(G)}^\ast \rightarrow M(G)$ is a continuous epimorphism using Proposition \ref{11}
	and part $( \rm{i})$, $G$ is finite.
	\\ $\rm{(iii)}$ By Remark \ref{7},  ${L^1(G)}^{\ast\ast}$ does not  have any nonzero continuous point derivation associated to any
	character in $ \Phi({L^1(G)}^{\ast\ast}).$ Hence $ G $ is finite by applying \cite[Theorem 11.17]{Dales.Lau.Strauss D}.
	\\$\rm{(iv)}$ By \cite[Theorem 2.11]{Lau.Pym}  $M(G)$ is isometrically isomorphic with $E\Box {L_0^\infty(G)}^\ast$ where $E$ is a
	norm one right identity of ${L_0^\infty(G)}$. Furthermore, mapping $ F\mapsto E\Box F $ is a continuous epimorphism from
	${L_0^\infty(G)}^\ast$ onto $M(G).$ Again, by applying Proposition \ref{11} and part $( \rm{i})$, $G$ is finite.
	\\$\rm{(v)}$ Since the restriction map from ${\dot{Y}}^\ast$ into $(AP(G))^\ast$ is a continuous epimorphism, $(AP(G))^\ast$ is
	essentially contractible by Proposition \ref{11}. On the other hand, we have the isomorphism $M(bG)\cong(AP(G))^\ast$, where $bG$ is the Bohr
	compactification of $G$. Hence, $bG$ must be finite, as established by part $(\rm{i})$. Since $G$ is a maximally almost periodic group,
	the canonical homomorphism from $G$ into $bG$ is injective. Therefore, $G$ must also be finite.
\end{proof}
%%%%%%%%%%%%%%%%%%%%%%%%%%%%%%
%%%%%%%%%%%%%%%%%%%%%%%%%%%%%%
%%%%%%%%%%%%%%%%%%%%%%%%%%%%%%
Ghahramani and Loy introduced the concept of essential amenability for Banach
algebras in \cite{Ghahramani.Loy} and explored various hereditary properties of essential amenable Banach algebras. Building on this work, our goal is to derive some hereditary properties specific to essential contractible Banach algebras.

Let $\{A_i\}_{i\in I}$ be a family of Banach algebras. The product
space $\underset{i\in I}{\prod} A_i$ contains all mapping $ a
\mapsto (a(i))_{i\in I} $ from $ I $ onto $ \bigcup_{i\in I}A_i. $
The $ l_p$-direct sum $(1\leq p<\infty)$  and $l_{\infty}$-direct sum of $\{A_i\}_{i\in I}$
are defined as following
$$\overset{p}{\underset{i\in I}{\bigoplus}}A_i := \Bigl\{ (a(i))_{i\in I}\in\underset{i\in I}{\prod} A_i; \|(a(i))_{i\in I}\|_{p}  = \Big(\sum_{i\in I}\|a(i)\|^p\Big)^{1/p}<\infty \Bigr\},$$
and
$$\overset{\infty}{\underset{i\in I}{\bigoplus}}A_i:=\Bigl\{ (a(i))_{i\in I}\in\underset{i\in I}{\prod}A_i;\|(a(i))_{i\in I}\|_{\infty}=\underset{i}{\max}\|a(i)\|< \infty \Bigr\}.$$
Moreover, the elements $(a(i))_{i\in I}$ in $ l_{\infty}$-direct sum of $\{A_i\}_{i\in I}$ for which $\lim_{I}a(i)=0$ is denoted by $\overset{0}{\underset{i\in I}{\bigoplus}}A_i$ and is calles the $c_0$-direct sum of $\{A_i\}_{i\in I}$.
The $c_0$-direct sum  and $l_p$-direct sum $(1\leq p\leq\infty)$  of the collections of Banach algebras by coordinatewise multiplication,
$\|\cdot\|_p$ and $\|\cdot\|_\infty$ form a Banach algebra, respectively.

Let $ A $ and $ B $ be  Banach algebras and $\theta\in \Phi(B).$
The Cartesian multiplication $ A\times B $ by the following multiplication
$$(a,b)(a^\prime,b^\prime) =(aa^\prime+\theta(b)a^\prime+\theta(b^\prime)a ,bb^\prime),$$
and the norm
$$\|(a,b)\| = \|a\|_A + \|b\|_B,$$
for all $ a,a^\prime\in A$ and $b,b^\prime\in B$, becomes a Banach algebra. Such product is called the $\theta$-Lau product of $A$ and $B$
and is denoted by $A\times_{\theta}B$. In the case of $B=\Bbb C$ and $\theta=id_{\Bbb C}$, then $A\times_{\theta}\Bbb C$ is denoted by
$A^{\sharp}$ and is called the unitization of $A$.
\begin{Proposition}\label{20} Let $\{A_i\}_{i\in I}$ be a family of Banach algebras and $A, B$ be Banach algebras.
	\begin{enumerate}
		\item[(i)] If $A$ is essentially contractible and $I$ is a closed two-sided ideal of $A,$ then $A/I$ is essentially contractible.
		\item[(ii)] Let $I$ be a closed two-ideal of $A$ with $A^n\subseteq I$ for some $n\geq2$. If $I$ is essentially contractible and $A/I$
		is contractible then $A$ is essentially contractible.
		\item[(iii)] If $\overset{p}{\underset{i\in I}{\bigoplus}}A_i$ $(1\leq p\leq\infty)$ or $\overset{0}{\underset{i\in I}{\bigoplus}}A_i$ is
		essentially contractible, then so is $A_i$ for each $i\in I.$
		\item[(iv)]If $A\times_{\theta}B$ is essentially contractible, then $A$
		and $B$ are essentially contractible. 
		\item[(v)] If $A^{\sharp}$ is essentially contractible then so is $A$.
	\end{enumerate}
\end{Proposition}
\begin{proof}
	$\rm{(i)}$ It is a direct consequence of Proposition \ref{11}.

	$\rm{(ii)}$ Let $D:A\rightarrow X$ be a bounded derivation and $ X $ be a neo-unital Banach $A$-bimodule. Since $A^n\subseteq I$ so
	$$X=A\cdot X\cdot A=A\cdot (A\cdot X\cdot A)\cdot A=A^2\cdot X\cdot A^2=A^2\cdot (A\cdot X\cdot A)\cdot A^2=\cdots=A^n\cdot X\cdot A^n\subseteq I\cdot X\cdot I\subseteq X.$$ This implies that $X$ is a neo-unital Banach $I$-bimodule.
	Since $D\big|_{I}:I\rightarrow X$ is a bounded derivation thus there
	exists $ x\in X $ such that $ D\big|_{I} = {ad}_x.$ Define
	$\widetilde{D}=D-{ad}_x$. Then $\widetilde{D}:A\rightarrow X$ is a
	bounded derivations such that $\widetilde{D}|_I =0,$ and so drops
	to a map, denoted by $\overline{D}$, from $A/I$ into X. i.e.,
	$\overline{D}(a+I)=\widetilde{D}(a)$. Set
	$$Y=\{x\in X:a x = x a = 0{~\rm{for~ all}~}a\in I\}.$$
	Then $Y$ is a closed submodule of $X$. Also $Y$ is a Banach $A/I$-bimodule by the usual module actions
	$$(a+I)\cdot y=ay\quad {\rm and}\quad y\cdot (a+I)=ya,\qquad (a\in A, y\in Y).$$
	Since $\widetilde{D}|_I =0,$ we obtain
	$$a\cdot\widetilde{D}(b)=\widetilde{D}(ab)-\widetilde{D}(a)\cdot b=0,\quad\widetilde{D}(b)\cdot a=\widetilde{D}(ba)-b\cdot\widetilde{D}(a)=0,\qquad (a\in I, b\in A).$$
	This means that $\overline{D}(A/I)\subseteq Y$. Therefore, there exists $y\in Y$ such that $\overline{D}={ad}_y$ and so $D={ad}_{x+y}$

	$\rm{(iii)}$ For every $i\in I,$ $A_i$  is a homomorphic image of
	$\overset{p}{\underset{i\in I}{\bigoplus}}A_i$ or
	$\overset{0}{\underset{i\in I}{\bigoplus}}A_i$, so it is essentially
	contractible using Proposition \ref{11}.
	
	$\rm{(iv)}$ Let $X$ be a neo-unital Banach $A$-bimodule and $D:A
	\rightarrow X$ be a continuous derivation. Then $X$ is a neo-unital Banach
	($A\times_{\theta}B$)-bimodule by the following module actions
	$$(a,b) \cdot x = a \cdot x + \theta(b)x,\quad x\cdot(a,b)  =   x\cdot a + \theta(b)x,\qquad (a\in A, b\in B, x\in X).$$
	We define $\widehat{D}:A\times_{\theta}B\rightarrow X$ by
	$\widehat{D}=D\circ p_A,$ where $p_A:A\times_{\theta}B\rightarrow A$ is the projection map on the first variable.
	Continuity of $D$ and $p_A$ implies that $\widehat{D}$ is continuous. On
	the other hand, we have
	\begin{equation*}
		\begin{split}
			\widehat{D}((a,b)(a^\prime,b^\prime))&=D\circ p_A((aa^\prime+\theta(b)a^\prime+\theta(b^\prime)a,bb^\prime))
			\\&=D(aa^\prime+\theta(b)a^\prime+\theta(b^\prime)a)
			\\&=aD(a^\prime)+D(a)a^\prime+\theta(b)D(a^\prime) +\theta(b^\prime)D(a)
			\\&=(a,b)\cdot D(a^{\prime})+D(a)\cdot(a^{\prime},b^{\prime})
			\\&=(a,b)\cdot\widehat{D}((a,b))+\widehat{D}((a,b))\cdot(a^\prime,b^\prime).
		\end{split}
	\end{equation*}
So $\widehat{D}$ is a derivation and then is inner. Therefore $D$ is
	inner and we conclude that $A$ is essential contractible. Similarly,
	$B$ is essentially contractible. 
	
	$\rm{(v)}$ Note that $A\times_{\theta}\mathbb{C} = A^{\sharp}.$
	So by part $\rm{(iv)}$ the proof is complete.
\end{proof}
It is important to note that the converse of Proposition \ref{20} (iii) does not hold.
For instance, the direct sum $\overset{\infty}{\underset{n\in\Bbb N}{\bigoplus}}\Bbb C=\ell^{\infty}$ is not essentially contractible.

Recall that an element $\delta$ in Banach algebra $A$ is called a central  idempotent if  $\delta^2= \delta \cdot\delta =  \delta$
and $ \delta \cdot a = a \cdot\delta$ for all $ a\in A. $
\begin{cor}
	Let $A$ be an essentially contractible Banach algebra with a central idempotent $\delta$. Then $A$
	has a contractible closed ideal.
\end{cor}
\begin{proof}
	Let $A$ be essentially contractible and $\delta$ be a central idempotent. Define the set
	$I=\{a\in A: a\delta=0\}$, which is a closed ideal in $A$. By applying Proposition \ref{20} $\rm{(i)}$,
	the quotient algebra $A/I$ is essentially contractible. Since $A/I\cong A\delta$, it follows that $A\delta$ is also essentially contractible. Furthermore, by Theorem \ref{3}, $A\delta$ is contractible, as it is unital with unit element $\delta$.
\end{proof}
%%%%%%%%%%%%%%%%%%%%%%%%%%%%%%%%%%%%%%%%%%%%%%%%%%%%%%%%%%%%%%%%%%%%%%%%%%%%%%%%%%%%%%%%%%%%%%%%%%%%%%%%%%%%%%%%%%%%%%%%%%%%%%%%%%%%%%%%%%%%%%%%%%%%%%%%%%
\section{Essential contractibility of Segal Banach algebras.}\label{sec4}
Segal algebras  are the certain classes of Banach algebra
associated with a locally compact group  $G$ that was introduced by
Reiter \cite{re}. For a locally compact group $G$, a Segal algebra
$S^1(G)$ is a left ideal of $L^1(G)$ that satisfies the following
conditions:
\\
$\rm{(i)}$ $\overline{S^1(G)}^{\|\cdot\|_1} = L^1(G);$\\
$\rm{(ii)}$ $(S^1(G), \|\cdot\|_{S^1})$ is a Banach space and $\|\cdot\|_{1}\leq\|\cdot\|_{S^1}$;\\
$\rm{(iii)}$ $S^1(G)$ is left translation invariant, and $\|L_xf\|_{S^1}=\|f\|_{S^1}$ for all $x\in G$ and every $f\in S^1(G)$,
where $L_xf$ is defined by $L_xf(y)=f(x^{-1}y)$ for all $y\in G$;\\
$\rm{(iv)}$ For every $ x\in G$ and  $f\in S^1(G),$ the map $ x \mapsto L_x f $ is continuous in $ \|\cdot\|_{S^1}$-norm.\\
A Segal algebra $S^1(G)$ is called symmetric if it also has the following conditions:
\\
$\rm{(v)}$ $S^1(G)$ is right translation invariant, and $\|R_xf\|_{S^1}=\|f\|_{S^1}$ for all $x\in G$ and every $f\in S^1(G)$,
where $R_xf$ is defined by $R_xf(y)=f(xy)$ for all $y\in G$;\\
$\rm{(vi)}$ For every $x\in G$ and  $f\in S^1(G),$ the map  $x \mapsto R_x f$ is continuous in $\|\cdot\|_{S^1}$-norm.

It is well-known that every Segal algebra has a left
approximate identity that is bounded in $\|\cdot\|_1$ see \cite{re}.

An abstract Segal algebra is a general  concept of Segal algebra. A dense left ideal  $ B $ of a Banach algebra $A$ is called
an abstract Segal algebra when it admits a norm $ \|\cdot \|_B $ for which $(B,\|\cdot\|_B)$ is a Banach algebra and \\
$\rm{(i)}$ there exists $M>0$ such that $\|b\|_{A}\leq M\|b\|_{B}$ for each $b\in B$.\\
$\rm{(ii)}$ there exists $C>0$ such that $\|ab\|_{B}\leq C \|a\|_{A}\|b\|_{B}$ for each $a\in A, b\in B$.

%Every Segal algebra is an abstract Segal algebra, while the converse is not valid necessarily.
\begin{Definition}
	Let $I$ be an ideal of  a Banach algebra $A$. We say that  $I$ has
	property $(\mathcal{P})$ if $I$ admits a norm $\|\cdot\|_I$ which $(I,
	\|\cdot\|_I)$ is a Banach algebra and
	$$\|ab\|_I, \|ba\|_I\leq \|a\|_A \|b\|_I,\qquad (a\in A, b\in I).$$
\end{Definition}
An abstract Segal algebra $B$ is called symmetric if it also has property $(\mathcal{P}).$
%%%%%%%%%%%%%%%%%%%%%%%%%%%%%%%%%%%%%%%%%%%%%%%%%%%%%%%%%%%%%%%%%%%%%%%%%%%%%

\noindent The famous theorem of the real analysis, the continuous extension theorem, can also be generalized to bounded derivations in the case where the range is a Banach space. In fact, the same technique works here. 
\begin{Proposition}\label{Prop}
	Let $A$ be a Banach algebra, $X$ be a Banach
	$A$-bimodule, $\mathcal{D}$ be a dense subalgebra of $A$ and $T:
	\mathcal{D}\rightarrow X$ be a bounded derivation. Then
	there exists a bounded operator $\tilde{T}:A \rightarrow X$ which
	is a derivation and $\tilde{T}\Big|_{\mathcal{D}} = T.$ 
	This extension is unique. Furthermore,  if $T$ is inner, then
	$\tilde{T}$ is inner.
\end{Proposition}
\begin{proof}
	Take $a\in A$ and choose $\{a_n\} \subset \mathcal{D}$ such that
	$a_n \rightarrow a$. Then $\{a_n\}$ is a Cauchy sequence in $A$ and so 
	$\{T(a_n)\}$ is a Cauchy sequence in $B$.
	Since $B$ is complete so $\{T(a_n)\}$ is convergent.
	We set $\tilde{T}(a):=\underset{n\rightarrow\infty}{\lim} T(a_n)$. If
	$a_n\rightarrow a$ and $b_n\rightarrow a$, using continuity of $T,$ we conclude that $\underset{n\rightarrow\infty}{\lim} T(a_n-b_n) =0.$
	It means that $\tilde{T}$ is well-defined.
	It is easy to see that $\tilde{T}$ is a linear derivation. Also, we have
	$$\|\tilde{T}(a)\| = \|\underset{n\rightarrow\infty}{\lim} T(a_n)\| = \underset{n\rightarrow\infty}{\lim} \|T(a_n)\| \leq \underset{n\rightarrow\infty}{\lim} \|T\|\|a_n\| = \|T\|\|a\|.$$
	This means that $\tilde{T}$ is bounded.
	Finally, if $\tilde{T_1}$ and $\tilde{T_1}$
	are two continuous extensions of $T,$ then for $a\in A$ we have
	$$\tilde{T_1}(a)=\underset{n\rightarrow\infty}{\lim}\tilde{T_1}(a_n)=\underset{n\rightarrow\infty}{\lim}T(a_n)=\underset{n\rightarrow\infty}{\lim}\tilde{T_2}(a_n)=\tilde{T_2}(a)$$
	Furthermore, if $T$ is inner, the uniqueness of extension and continuity of action imply that $\tilde{T}$ is inner.
	
\end{proof}
The extension obtained from Proposition \ref{Prop} gives us criterions for essential contractibility of Segal algebras and dense substructures.
\begin{Proposition}\label{4}
	Let $A$ be a Banach algebra, and $I$ be a dense subalgebra of $A$. If there exists an integer $n \geq 2$ such that $A^n  \subseteq I$ and  $I$ is
	essentially contractible, then $A$ is essentially contractible.
\end{Proposition}
\begin{proof}
	Let $X$ be a neo-unital Banach $A$-bimodule, and let $D:A\rightarrow X$ be
	a continuous derivation. It follows that $X=A\cdot (A\cdot X\cdot
	A)\cdot A$, and by induction, we obtain the inclusion $X = A^n\cdot X\cdot  A^n \subseteq
	I.X.I \subseteq X.$ This means that $X$ is a neo-unital Banach $I$-bimodule. Consequently, 
	there exists an element $x\in X$ such that $D|_{I} = {ad}_{x}.$ Since
	$D$ is the unique extension of $D|_{I}$ we conclude that $D$ is inner by
	Proposition \ref{Prop}.
\end{proof}
\begin{cor}
	Let $B$ be an essentially contractible abstract Segal algebra with respect to $A$. If $A^n \subseteq
	B$ for some integer $n\geq 2$ then $A$ is essentially contractible.
\end{cor}
Now, we are ready to present the main result that was proved amenability version of that in \cite{Ghahramani.Loy}.
\begin{Theorem}\label{1}
	Suppose that $A$ is  a contractible Banach algebra and $I$ is a
	non-closed ideal in $A$ with the property $(\mathcal{P})$. If $I$ has an approximate identity, which  is an approximate
	identity for $A$, then $I$ is essentially contractible.
\end{Theorem}
\begin{proof}
	Let $X$ be a neo-unital Banach $I$-bimodule and let $D:I\rightarrow X$ be
	a continuous derivation. Since $I$ has property $(\mathcal P)$, $X$ can be regarded as a Banach
	$A$-bimodule, as demonstrated in the proof of \cite[Theorem 7.1]{Ghahramani.Loy}. Hence, using Proposition \ref{Prop}, there exists a unique extension $\widetilde{D}:A
	\rightarrow X$. Since
	$A$ is contractible, $\widetilde{D}$ is an inner derivation and consequently $D$ is also
	inner.
\end{proof}
Applying Proposition \ref{4} and Theorem \ref{1} we obtain the following result.
\begin{cor}
	Suppose that $A$ is a Banach algebra. Let  $I$ be a non-closed ideal of $A$
	with property $(\mathcal{P})$ and it has an approximate identity which is an approximate identity for $A$.
	Then $A$ is essentially contractible if and only if $I$ is essentially contractible
\end{cor}
In the following, we discuss the essential contractibility of Segal algebras associated with groups and hypergroups. For further details on hypergroups, we refer to the work of Jewett \cite{Jewett}.
\begin{cor}\label{22} The following Banach algebras are essentially contractible:
	\item[(i)] A symmetric Segal algebra $S^1(G),$ for a finite group $G$;
	\item[(ii)] A Segal algebra $S^1(H),$ for a commutative finite hypergroup $H$;
	\item[(iii)] $\overline{[B^2]}^{\|\cdot\|_{ B}},$ for a symmetric abstract Segal algebra $B$ with respect to a contractible Banach algebra.
\end{cor}
\begin{proof}
	$\rm{(i)}$ By Theorem \ref{1} and \cite[Corollary 4.1.3]{Runde}.\\
	$\rm{(ii)}$ By Theorem \ref{1} and \cite[Proposition 3.2]{BS}.\\
	$\rm{(iii)}$ By Theorem \ref{1} and \cite[Theorem 3.5]{Samea}.
\end{proof}
%%%%%%%%%%%%%%%%%%%%%%%%%%%%%%%%%%%%%%%%%%%%%%%%%%%%%%%%%%%%%%%%%%%%%%%%%%%%%

Samea achieved several important results regarding abstract symmetric Segal algebras, as discussed in \cite{Samea}, which led to the following subsequent findings.
\begin{Lemma}\label{5}
	Let $A$ be a contractible Banach algebra and $B$ be a symmetric
	abstract Segal algebra with respect to $A.$ If $[B^2]$ is dense in
	$B$, then $B$ is essentially contractible.
\end{Lemma}
\begin{proof}
	Since $A$ is contractible, it is unital, and by \cite[Theorem
	3.5]{Samea}, $A$ has an approximate identity in $\overline{
		[B^2]}^{\|\cdot\|_B}$. Additionally, by \cite[Theorem 3.5]{Samea}
	symmetricality of $\overline{[B^2]}^{\|\cdot\|_B}$ implies that
	$\overline{[B^2]}^{\|\cdot\|_B}$ has the property $(\mathcal{P})$.
	Therefore, applying  Theorem \ref{1}, we conclude that $B=\overline{
		[B^2]}^{\|\cdot\|_B}$ is essentially contractible.
\end{proof}
\begin{Lemma}
	Let $A$ be a contractible Banach algebra and $I$ be a dense ideal in
	$A$ with the property $(\mathcal{P}).$ If $[I^2]$ is dense in $I$, then
	$I$ is essentially contractible.
\end{Lemma}
\begin{proof}
	Since $A$ is contractible, it is unital. Therefore, by \cite[Lemma
	4.5]{Samea} $I$ is a symmetric abstract Segal algebra. So by Lemma
	\ref{5}, $I$ is essentially contractible.
\end{proof}
\begin{cor}\label{symm}
	Every symmetric abstract Segal algebra with respect to a contractible Banach algebra with an approximate identity is essentially contractible.
\end{cor}
%%%%%%%%%%%%%%%%%%%%%%%%%%%%%%%%%%%%%%%%%%%%%%%%%%%%%%%%%%%%%%%%%%%%%%%%%%%%%
At the end, the aim is to provide an example of an (infinite dimensional) essentially contractible Banach algebra using matrix algebras.
Let $A$ be a  Banach algebra and $n\in\Bbb N$. The vector space of all $n\times n$-matrices $(a_{ij})$ with entries from $A$, denoted by $M_n(A)$, is a Banach algebra when equipped with the norm
$$\|(a_{ij})\|_{M_n(A)}=\sum_{i,j=1}^n\|a_{i,j}\|_A.$$
\begin{Lemma}\label{23}
	Let $A$ be a Banach algebra, $B$ be a symmetric abstract
	Segal algebra with respect to $A$.
	Then $M_n(B)$ is a symmetric abstract Segal algebra with respect to
	$M_n(A)$.
\end{Lemma}
\begin{proof}
	It is obvious that $M_n(B)$ is a dense left ideal in $M_n(A)$.
	First, let $a=(a_{i,j})$ and $b=(b_{i,j})$ be in $M_n(A)$ and $M_n(B)$, respectively.
	Then
	$$\|(b_{ij})\|_A=\sum_{i,j=1}^n\|b_{ij}\|_A\leq M\sum_{i,j=1}^n\|b_{ij}\|_B=M\|(b_{ij})\|_B$$
	Finally, let $(a_{ij})(b_{ij})=(c_{ik})$, where $c_{ik}=\sum_{j}a_{ij}b_{jk}$. Using the submultiplicativity property of $B$ we have
	$$\|c_{ik}\|_{B}=\left\|\sum_{j=1}^na_{ij}b_{jk}\right\|_B\leq\sum_{j=1}^n\|a_{ij}\|_A\|b_{jk}\|_B.$$
	Summing over all $i, k$, we obtain
	$$\|(c_{ik})\|_{M_n(B)}=\sum_{i,k=1}^n\|c_{ik}\|_B\leq\sum_{i,k=1}^n\sum_{j=1}^n\|a_{ij}\|_A\|b_{jk}\|_B.$$
	Reorganizing the sums gives:
	$$\|(c_{ik})\|_{M_n(B)}\leq\left(\sum_{i,j=1}^n\|a_{ij}\|_A\right)\left(\sum_{j,k=1}^n\|b_{jk}\|_B\right)=\|(a_{ij})\|_{M_n(A)}\|(b_{ij})\|_{M_n(B)}$$
	In a similar way we obtain
	$$\|(b_{ij})(a_{ij})\|_{M_n(B)}\leq\|(a_{ij})\|_{M_n(A)}\|(b_{ij})\|_{M_n(B)}$$
\end{proof}
In the sequel, we are prepared to present the example that was promised.
\begin{Example} Let $A$ be a contractible Banach algebra, $B$ be a symmetric abstract Segal algebra with respect to $A$ and $n\in\Bbb N$.
	By Lemma \ref{23}, $M_n(B)$ is a symmetric abstract
	Segal algebra with respect to $M_n(A)$. Using \cite[Lemma 3.5]{Samea} $
	\overline{[M_n(B)^2]}^{\|\cdot\|_{M_n(B)}}$ is a symmetric abstract Segal
	algebra with respect to $M_n(A),$ which has an approximate
	identity $(e_{\alpha})$ for $M_n(A)$. Moreover, by \cite[Theorem 3.2]{ess} $M_n(A)$ is contractible, Corollary \ref{symm} implies that $\overline{[M_n(B)^2]}^{\|\cdot\|_{M_n(B)}}$ is essentially contractible.
	Note that $
	\overline{[M_n(B)^2]}^{\|\cdot\|_{M_n(B)}}$ is not contractible; else it should have a unit element,
	namely $E$. Hence  we have
	$e_{\alpha}=e_{\alpha}E\rightarrow E$. Therefore,
	$e_{\alpha}I\rightarrow I$ and $e_{\alpha}I\rightarrow EI=E$, implies that $E=I$, where $I$ is the identity matrix. This would imply
	$M_n(A)=\overline{[M_n(B)^2]}^{\|\cdot\|_{M_n(B)}}\subseteq M_n(B)$, which is a contradiction.
\end{Example}

%%%%%%%%%%%%%%%%%%%%%%%%%%%%%%%%%%%%%%%%%%%%%%%%%%%%%%%%%%%%%%%%%%%%%%%%%%%%%%%%%%%%%%%%%%%%%%%%%%%%%%%%%%%%%%%%%%%%%%%%%%%%%%%%%%%%%%%%%%%%%%%%%%%%%%%%%%

\begin{small}

\end{small}
\end{document}